\documentclass[12pt,a4paper]{article}
\usepackage[centertags]{amsmath}
\usepackage{amsfonts}
\usepackage{calc}
\usepackage{tikz}
\usetikzlibrary{decorations.markings}
\tikzstyle{vertex}=[circle, draw, inner sep=0pt, minimum size=6pt]

\usepackage{graphics}
\usepackage{amssymb}
\usepackage{amsthm}
\usepackage{newlfont}
\usepackage{epsfig}
\usepackage{graphicx}
\usepackage{makeidx}
\theoremstyle{plain}
\newtheorem{thm}{Theorem}[section]

\newtheorem{lem}[thm]{Lemma}

\theoremstyle{definition}

\theoremstyle{remark} \tolerance=10000 \hbadness=10000
\def \ni{\noindent}

\author{Jismy Varghese \footnote{Email : kvjismy@gmail.com}\\ School of Computer Science\\
DePaul Institute of Science and Technology\\ Angamaly - 683 573\\ \vspace{0.3cm} Kerala, India.\\
Aparna Lakshmanan S.\footnote{E-mail :
aparnaren@gmail.com}\\
Department of Mathematics\\St.Xavier's College for Women\\Aluva -
683 101\\\vspace{0.2cm} Kerala, India.
}

\title{\textbf{Impact of Vertex Addition on Italian Domination Number}}

\date{}
\begin{document}
	
	\maketitle
	
	\begin{abstract}
		
		\ni\line(1,0){360}\\  An Italian dominating function (IDF), of a graph G is a function $ f: V(G) \rightarrow \{0,1,2\} $ satisfying the condition that for every $ v\in V(G) $ with $ f(v) = 0, \sum_{ u\in N(v)} f(u) \geq 2. $ The weight of an IDF on $G$ is the sum $ f(V)= \sum_{v\in V(G)}f(v) $ and Italian domination number, $ \gamma_I(G) $ is the minimum weight of an IDF. In this paper, we study the impact of corona operator and addition of twins on Italian domination number.
		   \\
		\ni\line(1,0){360}\\
		\ni {\bf Keywords:} Italian domination number, corona operator, twin vertex.  \\
		
		\ni {\bf AMS Subject Classification:} {\bf primary: 05C69}, {\bf secondary: 05C76.} \\
		\ni\line(1,0){360}\\
	\end{abstract}

\section{Introduction}

Let G be a simple graph with vertex set $V(G)$ and edge set $ E(G) $. If there is no ambiguity in the choice of G then we write $V(G)$ and $ E(G) $ as $ V $ and $ E $ respectively. The open neighborhood of a vertex $v \in V $ is the set $ N(v)= \{u: uv \in E\} $ and its closed neighborhood is $ N[v] = N(v) \cup \{v\}.$    A subset $ S \subseteq V$ of vertices in a graph is called a dominating set if every $ v\in V$ is either an element of $S$ or is adjacent to an element of $S$ \cite{Tha}. The domination number, $\gamma(G) $ is the minimum cardinality of a dominating set of G. \\

  An Italian dominating function, in short IDF, of a graph G is a function $f: V \rightarrow \{0,1,2\} $ which satisfies the condition that for every $ v\in V $ with $ f(v) = 0,$ $ \sum_{ u\in N(v)} f(u) \geq 2, $   i.e; either $v$ is adjacent to a vertex $u$ with $ f(u) = 2 $ or to at least two vertices $x$ and $y$ with $ f(x) = f(y) = 1 $. The weight of an Italian dominating function is $f(V)=\sum_{u \in V}f(u)$. The Italian domination number, $ \gamma_I(G) $ is the minimum weight of an Italian dominating function. An IDF with weight $ \gamma_I(G) $ is called $ \gamma_I $-function. Let $ V_i^f $ or simply $ V_i $, denote the set of vertices assigned $ i $ by the function $f$. The Italian domination number was first introduced in  \cite{Mustha} with the name Roman- \{2\}-domination. For any graph $ G $, the Italian domination number is bounded by $\gamma(G)\leq \gamma_I(G) \leq\gamma_R(G)\leq 2\gamma(G)  $ which was given in \cite{Mustha, Mac}. M. A. Henning and W. F. Klostermeyer  studied the Italian domination number of trees \cite{Henn}. The Italian domination number of generalized Petersen graph, $ P(n,3) $ is found in \cite{Gao}. In \cite{Jis}, it is proved that $ \gamma_I(G)+1 \leq \gamma_I(M(G))\leq \gamma_I(G)+2,  $ where $ M(G) $ is the Mycielskian graph of $ G. $ It is also proved that $ \gamma_I(S(K_n,2)) =2n-1 $ and $ n^{t-2}\alpha(G)\gamma_I(G)\leq \gamma_I(S(G,t)) \leq n^{t-2}(n-\gamma_I(G)-|V_2|-E_2) $ where $ S(G,t)  $ is the Sierpinski graph of $ G $, $ \alpha(G) $ is the  independence number of $ G $ and $ E_2 $ is the set of non-isolated vertices in $ <V_2>  $ \cite{Jis}. To know more about Italian domination number, the interested readers can refer \cite{Ali}, \cite{Gao2}, \cite{Haji}, \cite{Hay}, \cite{Ji}, \cite{Pour}.\\

 The corona of two graphs $G_1 = (V_1;E_1)$ and $G_2 = (V_2;E_2)$, denoted by $G_1\odot G_2,$ is the graph obtained by taking one copy of $G_1$ and $\lvert V_1\lvert$ copies of $G_2$, and then joining the $i^{th}$ vertex of $G_1$ to every vertex in the $i^{th}$ copy of $G_2$. We denote the complete graph on $n$ vertices by $ K_n $. If $G_2$ is $K_1$, then $G_1\odot G_2$ is $G_1$ together with one pendent vertex each attached to all the vertices of $G_1$. A false twin of a vertex $ u $ is a vertex $ u' $ which is adjacent to all vertices in $ N(u)$. A true twin of a vertex $ u $ is a vertex $ u' $ which is adjacent to all vertices in $ N[u] $. Two vertices $ u $ and $ u' $ are said to be twins if either they are true twins or false twins. For any graph theoretic terminology and notations not mentioned here, the readers may refer to \cite{Bal}.\\

The following result is useful in this paper.\\
\begin{thm} \cite{Mustha} \label{thm1}
	For the class of $P_n$, $ \gamma_I(P_n)= \lceil \frac{n+1}{2}\rceil. $	
\end{thm}

\section{Corona Operator on Italian Domination}

In this section, we find the value of Italian domination number of corona operator of any two graphs $G$ and $H$, where $H \ncong K_1$. The upper and lower bounds for $\gamma_I(G \odot K_1)$ is obtained and the corresponding realization problem is also settled. Also, the exact value of $\gamma_I(G \odot K_1)$, when $G$ belongs to some special classes of graphs is obtained.

\begin{lem} \label{lem1}
	Let $ G $ be a graph and $ u $ be a pendent vertex of $ G. $ Then there exists a $ \gamma_I$-function $ f $ of $ G $ in which $ f(u) \neq 2. $
\end{lem}
\begin{proof}
	If possibe assume that there exists a $ \gamma_I $-function with $ f(u) = 2. $ Note that  the weight of neighbor of $ u, $ say $ v $ is zero, due to the minimality of $ f. $ Then we can reassign $ f(u)=f(v)=1 $ or $f(u)=0, f(v)=2,  $ which is again a $ \gamma_I $-function on $ G $ with $ f(u) \neq 2. $ Hence the lemma.
\end{proof}

In this paper, here onwards, we consider $ \gamma_I$-functions for which the weight of a pendent vertex is not equal to 2.

\begin{thm} \label{thm2}
	For every graph $G$ and $H \ncong K_1,\ \gamma_I(G\odot H)=2n,$ where $n$ is the order of $G$.
\end{thm}

\begin{proof}
Define an Italian domination function $f$ of $ G\odot H $ as follows.
\[
f(v) =
\begin{cases}
2,\  \ for\ v\in V(G),\\
0,\  \ otherwise.\\
\end{cases}
\]
Therefore, $ \gamma_I(G\odot H)\leq 2n.$ There are $n$ mutually exclusive copies of $H$ each of which requires at least weight 2 in IDF. So $ \gamma_I(G\odot H) \geq 2n.$ Hence the theorem.

\end{proof}

\begin{thm} \label{thm3}
	For any graph $G$, $n+1 \leq \gamma_I(G\odot K_1)\leq 2n.$
	
\end{thm}
\begin{proof}
	Let $V(G)= \{v_1,v_2,...v_n\}$ and let $ u_i $ be the leaf neighbor of $ v_i $ in $ G\odot K_1. $ Define an IDF of $ G\odot K_1 $ as follows.
	\[
	f(u) =
	\begin{cases}
	2,\  \ for\ u=u_i,\\
	0,\  \ otherwise.\\
	\end{cases}
	\]
	Therefore, $ \gamma_I(G)\leq 2n. $ To prove the left inequality, let $ f $ be any IDF of $ G\odot K_1. $ By Lemma \ref{lem1} each $ u_i $ must be either in $ V_1  $ or adjacent to a vertex in $ V_2. $ If $ u_i \in V_1, $ for all $ i=1,2,3,...n, $ none of the vertices in $ G $ can be italian dominated by $ u_i $ alone. Therefore, $ f(V) \geq n+1. $ If $ u_i \notin V_1$ for some $ i $, then $ u_i $ is adjacent to a vertex in $ V_2 $ which further increases the value of $ f(V). $ Hence the theorem.
	
\end{proof}
\begin{thm} \label{thm4}
	Any positive integer $ a $ is realizable as the Italian domination number of $ G\odot K_1, $ for some $ G $ if and only if $ n+1 \leq a \leq 2n,  $ where n is the number of vertices in $ G. $
\end{thm}
\begin{proof}
 Let $ G $ be a graph with $ |V(G)| = n $. If $ \gamma_I(G\odot K_1)= a $ then by theorem \ref{thm3}, $ n+1\leq a \leq 2n. $
 To prove the converse, let $ G $ be the graph $ K_{1,m} \cup (n-m-1)K_1$ where $ 0\leq m \leq n-1. $ Let $ v_1,v_2,...v_{m+1} $ be the vertices of $ K_{1,m} $ in which $ v_1 $ is the universal vertex and $ v_{m+2}, v_{m+3},...,v_n $ be the isolated vertices in $ G. $ Let $ v_i' $ be the leaf neighbor of $ v_i $ in $ G\odot K_1. $ Define an IDF $ f $ on $ V(G\odot K_1) $ as follows.
  \[
  f(u) =
  \begin{cases}
  2,\  \ for\ u=v_i,\ i=1,m+2,m+3,...,n,\\
  1,\ \ for\ u=v_i',\ i=2,3,...,m+1,\\
  0,\  \ otherwise.\\
  \end{cases}
  \]
  Therefore, $ f $ is a $ \gamma_I $-function with weight $ 2(n-m)+m=2n-m,\ 0\leq m \leq n-1. $ So $ \gamma_I(G\odot K_1) $ varies from $ n+1 $ to $ 2n. $ Hence the theorem,
  \end{proof}

\begin{thm} \label{thm5}
$ 	\gamma_I(G\odot K_1) = n+1$ if and only if  $ G $ has a universal vertex.
\end{thm}

\begin{proof}
	Let $ G $ be a graph with vertices $ v_1,v_2,v_3,...v_n $ and let $ u_i $ be the leaf neighbor of $ v_i $. Let $ v_1 $ be the universal vertex in $ G. $ Define an IDF of $ G\odot K_1 $ as follows.
	\[
	g(v) =
	\begin{cases}
	2,\  \ \ v=v_1,\\
	1,\ \   \ v= u_i, i=2,3,...,n,\\
	0,\  \ otherwise.\\
	\end{cases}
	\]
	Then $ g(V)=n+1 $ which is the minimum possible and hence $ 	\gamma_I(G\odot K_1) = n+1.$\\
	
	To prove the converse part, assume that $ \gamma_I(G\odot K_1) = n+1.$ Let $ f $ be a $ \gamma_I $-function of $ G\odot K_1. $ If possible assume that $ G $ does not have a universal vertex. Out of $ n $ pendent vertices in $ G\odot K_1 $, let $ k $ vertices be in $ V_1^f $ so that the remaining $ n-k $ pendent vertices are  adjacent to vertices in $ V_2^f. $ Then $ f(V)= n+1 \geq k+2(n-k) = 2n-k. $ Therefore, $ k\geq n-1. $ If $ k=n-1 $ then there exists a pendent vertex $ u_l $ which is adjacent to a vertex in $ V_2^f. $  If $ u_l $ is adjacent to a vertex in $ V_2^f $ then, since $ u_l $ is not a universal vertex we need more vertices with non zero weight to Italian dominate vertices in $ G $, which is a contradiction to the fact that $ \gamma_I(G\odot K_1) = n+1.  $ Therefore our assumption is wrong. Hence $ G $ has a universal vertex.
 \end{proof}
\begin{thm} \label{thm6}
$	\gamma_I(G\odot K_1) = 2n$ if and only if $ G= K_n^c. $
\end{thm}
\begin{proof}
	Let $ v_1,v_2,...v_n $ be the vertices of $ G $ and let  $ u_i $ be the pendent vertex adjacent to $ v_i  $ in $ G\odot K_1 $ for $ i=1,2,...,n. $ If possible assume that there exists an edge $ v_iv_j $ in $ G. $ Then $ u_iv_iv_ju_j $ is a $ P_4 $ in $ G\odot K_1 $ which can be Italian dominated by assigning 2 to $ v_i $ and 1 to $ u_j. $ Now, assigning 2 to every $ v_k $ for $ k=1,2,...,n $ and $ k\neq i,j $ gives an IDF of $ G\odot K_1 $ with weight $3+2(n-2) = 2n-1$, which contradicts the fact that $ \gamma_I(G\odot K_1) = 2n . $  Hence $ G  $ does not have an edge. i.e., $ G=nK_1= K_n^c. $ It is trivial that if $ G= K_n^c $ then $ \gamma_I(G\odot K_1) = 2n . $
\end{proof}

\begin{thm} \label{thm7}
	\[
	\gamma_I(K_{p,q}\odot K_1)=
	\begin{cases}
	p+q+1,\ p=1\ or\ q=1,\\
	p+q+2,\ otherwise.
	\end{cases}
	\]
	
	\end{thm}
\begin{proof}
Let $ V(K_{p,q})= {u_1,u_2,...u_p,v_1,v_2,...,v_q} $ and $ u_i' $ be the leaf neighbor of $ u_i,\ i=1,2,...,p $ and $ v_j'  $ be that of $ v_j $ for $ j=1,2,...q $ in $ K_{p,q} \odot K_1. $ By the left inequality of \ref{thm3} $ p+q+1 \leq  \gamma_I(K_{p,q}\odot K_1). $\\

\ni{\bf Case1:} $p=1$ or $ q=1.$\\
Without loss of generality, let $ p=1. $ Define an IDF of $ K_{p,q}\odot K_1 $ as follows.
\[
f(u) =
\begin{cases}
2,\  \ for\ u=u_1,\\
1,\ \ for\ u=v_j',\ j=1,2,3,...,q,\\
0,\  \ otherwise.\\
\end{cases}
\]
The weight  $ f(V)= 2+q = p+q+1. $ Therefore, $\gamma_I(K_{p,q}\odot K_1) \leq p+q+1.$ Hence $\gamma_I(K_{p,q}\odot K_1) = p+q+1.$\\

\ni{\bf Case2:} $ p,q \geq 2. $\\
Define an IDF of $ K_{p,q}\odot K_1 $ as follows.
\[
f(u) =
\begin{cases}
2,\  \ for\ u=u_1\ and \ u=v_1,\\
1,\ \ for\ u=u_i'\ i=2,3,...,p,\ u=v_j',\ j=2,3,...,q,\\
0,\  \ otherwise.\\
\end{cases}
\]
The weight $ f(V)= 4+p-1+q-1 = p+q+2. $ Therefore, $\gamma_I(K_{p,q}\odot K_1) \leq p+q+2.$\\

 To prove the reverse inequality, if possible assume that there exists an IDF $ g $ of $ K_{p,q}\odot K_1 $ with weight $ p+q+1. $ Out of $p+q$ pendent vertices in $ K_{p,q} \odot K_1 $, let $ k $ vertices be in $ V_1^g. $ Note that, by Lemma \ref{lem1} we can always find a $ \gamma_I $-function in which pendent vertices are assigned values either 0 or 1. So that the remaining $ p+q-k $  pendent vertices are  adjacent to vertices in $ V_2^g. $ Hence the weight of $g$, $g(V)= p+q+1 \geq k+2(p+q-k).$ Hence, $ k\geq p+q-1. $ If $ k> p+q-1 $ then $ k=p+q $ so that all the pendent vertices are in $ V_1^g $ and none of them can be Italian dominated by any of the non-pendent vertices. Therefore, we need more vertices having non zero values under $ g $, which contradicts $ g=p+q+1. $ If $ k=p+q-1, $ then one pendent vertex, say $ x $, is adjacent to a vertex in $ V_2^g  $, say $ y $. Then $ y $ can not Italian dominate any of the vertices in its partite set of $ K_{p,q} $ containing $ y. $ Therefore, we need more vertices having non-zero values under $ g $ which is a contradiction. Hence the theorem.
\end{proof}

\begin{thm} \label{thm8}
	For any graph $ G,  \gamma_I((G\odot K_1)\odot K_1) = 3n $ where $ n=|V(G)|. $
\end{thm}

\begin{proof}
	Let $ G $ be a graph with vertex set $ V(G)= {v_1,v_2,...,v_n} $ and let $ u_i $ be the leaf neighbor of $ v_i $ in $ G\odot K_1. $ Let $ v_i' $ and $ u_i' $ be the leaf neighbors of $ v_i $ and $ u_i $ respectively, in $ (G\odot K_1)\odot K_1. $ There are $ n $ vertex disjoint $ P_4's,\  \ v_i'v_iu_iu_i'$ for $ i=1,2,...,n $ in $ (G\odot K_1)\odot K_1. $ Let $ f $ be an IDF on $ (G\odot K_1)\odot K_1. $ Then the 2 pendent vertices $ v_i' $ and $ u_i' $ in each $ P_4 $ should be either in $ V_1^f $ or adjacent to a vertex in $ V_2^f. $ If all the pendent vertices are in $ V_1^f, $ then to Italian dominate non-pendent vertices $ v_i $ and $ u_i $ we need more vertices with non-zero weight in $ P_4. $ The pendent vertices have no common neighbors. Hence, under $ f $, the sum of the values of vertices in each $ P_4 $ must be at least 3. Therfore, $ f(V) \geq 3n. $
	To prove the reverse inequality, define $ g $ as follows.
	\[
	g(u) =
	\begin{cases}
	1,\  \ for\ u=v_i',u_i',u_i,\ \ i= 1,2,...,n,\\
	0,\  \ otherwise.\\
	\end{cases}
	\]
	Then $ g $ is an IDF on $ (G\odot K_1)\odot K_1  $ with $ g(V) =3n. $ Hence the theorem.
\end{proof}

\begin{thm}
	
$	\gamma_I(P_n \odot K_1) =   \lceil \frac{4n}{3} \rceil. $

\end{thm}

\begin{proof}
Let $ v_1,v_2,...,v_n $ be the vertices of $ P_n $ and let $ u_i $ be the pendent vertex corresponding to $ v_i $, for $ i=1,2,...n $. If possible assume that there exists a $ \gamma_I$-function $ g $ of $ P_n \odot K_1  $ such that $ g(V)<\frac{4n}{3}. $ Note that we can always find a $ \gamma_I$-function in which pendent vertices are assigned values either $ 0 $ or $ 1 $ by Lemma \ref{lem1}. Out of $ n $ pendent vertices in $ P_n \odot K_1 $ let $ p $ vertices be in $ V_1^g, $ so that the remaining  $ n-p $ pendent vertices are assigned value $ 0 $ and hence adjacent to vertices in $ V_2^g. $ i.e., $ n-p $ vertices in $ P_n $ are assigned the value $ 2. $ These $ n-p $ vertices can Italian dominate atmost $ 3(n-p) $ vertices of $ P_n. $ i.e., at least $ n-(3(n-p)) = 3p-2n $ vertices are not yet Italian dominated. To Italian dominate these $ 3p-2n $ vertices we need atleast $ \frac{3p-2n}{3} $ more vertices of weight $ 1 $ in $ P_n. $ Therefore, $ g(V) > p+2(n-p)+\frac{3p-2n}{3}.$ i.e., $ g(V)> \frac{4n}{3}$, which is a contradiction. So $ g(V)\geq \frac{4n}{3}. $
 Define an IDF of $ P_n \odot K_1  $ as follows.\\
When $ n=3k. $
\[
f(u) =
\begin{cases}
2,\  \ if\ u=v_{3j-1},\ for\ j=1,2,...,k,\\
1,\ \ if\ u=u_j,\ for\ every\ j\ such\ that\ f(v_j) \neq 2,\\
0,\  \ otherwise.\\
\end{cases}
\]
Then $ f $ is an IDF with $ f(V)= \frac{4n}{3}. $\\
So that, $ \gamma_I(P_n \odot K_1) \leq \frac{4n}{3}. $ Therefore, $ \gamma_I(P_n \odot K_1) = \frac{4n}{3}. $ \\
When $ n=3k+1. $
\[
f(u) =
\begin{cases}
2,\  \ if\ u=v_n, \ or\ v_{3j-1}\ for\ j=1,2,...,k,\\
1,\ \ if\ u=u_j,\ for\ every\ j\ such\ that\ f(v_j) \neq 2,\\
0,\  \ otherwise.\\
\end{cases}
\] 	
Then $ f $ is an IDF with $ f(V)= \frac{4n+2}{3}. $ So that, $ \gamma_I(P_n \odot K_1) \leq \frac{4n+2}{3}. $ Therefore, $ \gamma_I(P_n \odot K_1) = \frac{4n+2}{3}. $ \\
When $ n=3k+2. $
\[
f(u) =
\begin{cases}
2,\  \ if\ u=v_{3k+2},\ k=0,1,2,...\\
1,\ \ if\ u_j \ with\ f(v_j) \neq 2,\\
0,\  \ otherwise.\\
\end{cases}
\] 	
Then $ f $ is an IDF with $ f(V)= \frac{4n+1}{3}. $ So that, $ \gamma_I(P_n \odot K_1) \leq \frac{4n+1}{3}. $ Therefore, $ \gamma_I(P_n \odot K_1) = \frac{4n+1}{3}. $ \\
Hence, $\gamma_I(P_n \odot K_1) =   \lceil \frac{4n}{3} \rceil. $

\end{proof}

\begin{thm}
$ \gamma_I(C_n \odot K_1) = \lceil \frac{4n}{3} \rceil. $
	
\end{thm}
\begin{proof}
The proof is similar to that of $ P_n. $
\end{proof}

\section{Addition of twin vertex}
In this section, we discuss the impact of addition of twin  vertices to a graph $ G $ on the Italian domination number of a graph.

\begin{lem} \label{lem2}
	Let $ u $ and $ u' $ be true twins in a graph G. Then there exists a $ \gamma_I$-function of $ G $ in which $ f(u')=0$. 
\end{lem}
	\begin{proof}
		Let $ f $ be a $ \gamma_I $-function of G. If $ f(u')=2, $ then $ f(u)=0, $ due to the minimality of $ f. $ Now, reassign $ f(u)=2 $ and $ f(u')=0 ,$ so that $ f $ is a $ \gamma_I- $function of $ G $ with the required property.\\
		
		If $ f(u')=1 $ then $ f(u) $ can be either $ 1 $ or $ 0. $ Note that due to the minimality it can not be $ 2. $ If $ f(u)=1 $ then we can reassign $ f(u)=2 $ and $ f(u')=0 $ so that $ f $ is still a $ \gamma_I $-function of $ G $ with the required property. If $ f(u)=0 $ then to Italian dominate $ u $ there esiats a $ v \in N(u) $ such that $ f(v)=1. $ Since $ N(u)=N(u'), $ in this case also we can interchange the weights of $ u $ and $ u' $ to get a $ \gamma_I$-function in which $ f(u')=0. $ Hence the lemma.
\end{proof}
\begin{thm}
	Let $ G $ be a graph and $ u \in V(G). $ Let $ H $ be the graph obtained from $ G $ by attaching a true twin $ u' $ to $ u. $ Then $ \gamma_I(H)=\gamma_I(G) $ or $ \gamma_I(G) +1$.
\end{thm}
\begin{proof}
	Let $ f $ be a $ \gamma_I $-function of $ G. $ If $ u \in V_0^f \cup V_2^f $ then $ f $ can be extended to an IDF of $ H $ by assigning $ 0 $ to $ u' $ so that
	\begin{equation} \label{a}
	\gamma_I(H) \leq \gamma_I(G).
	\end{equation}
	
	If $ u \in V_1^f $ and there exists $ v \in N(u) $ such that weight of $ v $ not equal to $ 0 $ then $ f $ can be extended to an IDF of $ H $ by assigning $ 0 $ to $ u' $ so that
	\begin{equation} \label{b}
	\gamma_I(H) \leq \gamma_I(G).
	\end{equation}
	Now assume that there does not exist any $ \gamma_I$-function of $ G $ for which  $ u \in V_0^f \cup V_2^f $ or $ |N(u)\cap (V_1^f\cup V_2^f)| \geq 1, $ then we can extend $ f $ to an IDF of $ H $ by assigning $ 1 $ to $ u' $ so that
	\begin{equation} \label{c}
	\gamma_I(H) \leq \gamma_I(G) +1.
	\end{equation}
	Let $ g $ be a $ \gamma_I $-function of $ H. $ Then by Lemma \ref{lem2} there exists an IDF $ g $ in which $ g(u')=0. $ Then the restriction of $ g $ to $ V(G) $ is an IDF of $ G $ so that
	\begin{equation} \label{d}
	\gamma_I(G)\leq \gamma_I(H)
	\end{equation}
	The weight of $ u $ in $ H $ can be $ g(u)= 0,1 $ or $ 2. $ If $ g(u)=2, $ all the vertices in the neighborhood of $ u $ other than $ u' $ are Italian dominated by some other vertices in $ H. $ Then the restriction of $ g $ to $ G $ by assigning weight $ 1 $ to $ u $ is an IDF of $ G. $ Therefore, $ \gamma_I(G)\leq \gamma_I(H)-1. $ i.e.,
	\begin{equation} \label{e}
	\gamma_I(G) +1\leq \gamma_I(H).
	\end{equation}
	 From equations \eqref{a}, \eqref{b}, \eqref{c}, \eqref{d} and \eqref{e}, we get,
	 $ \gamma_I(H)=\gamma_I(G) $ or $ \gamma_I(G) +1.$
\end{proof}
\begin{lem} \label{lem3}
	Let $ u $ and $ u' $ be  false twins in a graph $ G. $ Then there exists a $ \gamma_I$-function $ f $ of $ G $ in which $ f(u')\neq 2. $
\end{lem}
\begin{proof}
	Let $ f $ be a $ \gamma_I $-function with $ f(u')=2. $ Then $ f(u)=0, $ by the minimality of $ f. $ If $ f(u)=0 $ then there exists a $ v \in N(u) $ such that $ f(v)=2 $ or two vertices $ x,y \in N(u) $ such that $ f(x)=f(y)=1. $ Since, $ u $ and $ u' $ have the same neighborhood, exchange weights of $ u $ and $ u'. $ Then we get a $ \gamma_I$-function with same weight and $ f(u')=0. $
\end{proof}
\begin{thm}
	Let $ G $ be a graph and $ u\in V(G). $ Let $ H $ be a graph obtained from $ G $ by attaching a false twin $ u' $ to $ u. $ Then $ \gamma_I(H)=\gamma_I(G) $ or $ \gamma_I(G) +1.$
\end{thm}
\begin{proof}
	Let $ f $ be a $ \gamma_I$-function of $ G. $ If $ u \in V_0^f $ or $ |N(u)\cap V_2^f|\geq 1 $ or $ |N(u)\cap V_1^f| \geq2 $ then $ f $ can be extended to an IDF of $ H $ by assigning $ 0 $ to $ u' $ so that
	\begin{equation} \label{1}
	 \gamma_I(H)\leq \gamma_I(G).   \\
	\end{equation}
	Now, assume that there does not exist any $ \gamma_I$-function of $ G $ for which any of the above conditions are satisfied. Then we can extend $ f $ to an IDF of $ H $ by assigning $ 1 $ to $ u' $ so that
	\begin{equation} \label{2}
	\gamma_I(H) \leq \gamma_I(G)+1. \\
	\end{equation}
	Let $ g $ be a $ \gamma_I$-function of $ H. $ Then by Lemma \ref{lem3} there exists a $ \gamma_I$-function with $ g(u')\neq 2. $ Therefore, $ g(u')=1 $ or $ 0. $ If $ g(u')=0 $ then the restriction of $ g $ to $ G $ is an IDF of $ G.  $ Therefore,
	\begin{equation} \label{3}
	  \gamma_I(G)\leq \gamma_I(H) \\
	\end{equation}
	If $ g(u')=1 $, but all the neighbors of $ u' $ are Italian dominated by some other vertices (i.e., $ u' $ is assigned value $ 1 $ to Italian dominate itself), then the restriction of $ g $ to $ G $ will be an IDF with $ \gamma_I(G)\leq \gamma_I(H)-1. $ i.e., \begin{equation} \label{4}
	\gamma_I(G)+1 \leq \gamma_I(H)  .
	\end{equation}  From equations \eqref{1}, \eqref{2}, \eqref{3} and \eqref{4}, we get,
	$ \gamma_I(H)=\gamma_I(G) $ or $ \gamma_I(G) +1.$
\end{proof}

\section{Conclusion}

In this paper, the impact of corona operator and addition of twins on the Italian domination number is studied. The following problems may also be worth investigating.\\

\ni{\bf Problem 1:} The effect of other graph operations on Italian domination number.\\

\ni{\bf Problem 2:} The effect of corona operator on Italian domination number of some other graph classes.\\

\ni{\bf Problem 3:} The effect of corona operator on other domination parameters.\\

{}

\end{document}